\documentclass[letterpaper, 10 pt, conference]{ieeeconf}  

\IEEEoverridecommandlockouts                              
\usepackage{cite}
\usepackage{amsmath,amssymb,amsfonts}
\usepackage{algorithmic}
\usepackage{graphicx}
\usepackage{textcomp}
\usepackage[utf8]{inputenc}
\usepackage{float}
\usepackage{algorithm}
\usepackage[left=2cm,right=2cm, top=2cm, bottom=2cm]{geometry}
\usepackage{tikz,pgfplots}
\usetikzlibrary{patterns}
\usetikzlibrary{quotes,angles}
\usepackage{url}
\usepackage{courier}

\usetikzlibrary{arrows.meta}

\newtheorem{lemma}{Lemma}
\newtheorem{corollary}{Corollary}
\newtheorem{theorem}{Theorem}
\newtheorem{assumption}{Assumption}

\title{\LARGE \bf
A Feasible Sequential Linear Programming Algorithm with Application to Time-Optimal Path Planning Problems
}

\author{David Kiessling$^1$, Andrea Zanelli$^2$, Armin Nurkanovi\'c$^3$, Joris Gillis$^1$,\\ Moritz Diehl$^{3,\,4}$, Melanie Zeilinger$^2$, Goele Pipeleers$^1$, Jan Swevers$^1$
\thanks{This work has been carried out within the framework of Flanders
Make SBO DIRAC: DIRAC - Deterministic and Inexpensive Realizations of Advanced
Control. Flanders Make is the Flemish strategic research centre for the manufacturing industry.}
\thanks{$^1$ Department of Mechanical Engineering, KU
Leuven, and DMMS lab, Flanders Make, Leuven, Belgium
        {\tt\small david.kiessling@kuleuven.be}}%
\thanks{$^2$ Institute for Dynamic Systems and Control, ETH Zurich, Switzerland}
\thanks{$^3$ Department of Microsystems Engineering (IMTEK), University of Freiburg, 79110 Freiburg, Germany
        }%
\thanks{$^4$ Department of Mathematics, University of Freiburg, 79110 Freiburg, Germany}
}

\begin{document}
\maketitle
\thispagestyle{empty}
\pagestyle{empty}

\begin{abstract}
In this paper, we propose a Feasible Sequential Linear Programming (FSLP) algorithm applied to time-optimal control problems (TOCP) obtained through direct multiple shooting discretization. This method is motivated by TOCP with nonlinear constraints which arise in motion planning of mechatronic systems. The algorithm applies a trust-region globalization strategy ensuring global convergence. For fully determined problems our algorithm provides locally quadratic convergence. Moreover, the algorithm keeps all iterates feasible enabling early termination at suboptimal, feasible solutions. This additional feasibility is achieved by an efficient iterative strategy using evaluations of constraints, i.e., zero-order information. Convergence of the feasibility iterations can be enforced by reduction of the trust-region radius. These feasibility iterations maintain feasibility for general Nonlinear Programs (NLP). Therefore, the algorithm is applicable to general NLPs. We demonstrate our algorithm's efficiency and the feasibility update strategy on a TOCP of an overhead crane motion planning simulation case.
\end{abstract}

\section{Introduction}
\label{sec:introduction}
We aim at solving Nonlinear Programs (NLP) arising in time-optimal motion planning of mechatronic systems. The TOCP we are looking at in this paper are obtained by the direct multiple shooting discretization \cite{Bock1984} and written as follows
\begin{subequations}
\label{OCP}
\begin{align}
\min_{\substack{x_0,\ldots,\,x_N \\ u_0,\ldots,\,u_{N-1} \\ s_0,\,s_N,\,T}} &T+ \mu_0^{\top} s_0 +\mu_N^{\top} s_N\\
\mathrm{s.t.}\quad &-s_0 \leq x_0-\overline{x}_0 \leq s_0,\\
\quad &x_{k+1}=f(x_k,\,u_k,\,\tfrac{T}{N}), &&\hspace{-6.5mm}k=0,\dots,\,N-1,\\
\quad &u_k \in \mathbb{U}_k, &&\hspace{-6.5mm}k=0,\dots,\,N-1,\\
\quad &x_{k} \in \mathbb{X}_{k}, &&\hspace{-6.5mm}k=0,\dots,\,N,\\
\quad &e(x_k,\,u_k)\leq 0, &&\hspace{-6.5mm}k=0,\dots,\,N-1,\\
\quad &-s_N \leq x_N-\overline{x}_N \leq s_N.
\end{align}
\end{subequations}
where $x_k\in\mathbb{R}^{n_x},\,  u_k\in\mathbb{R}^{n_u},\,s_0,\,s_N\in\mathbb{R}^{n_x}$ denote the state, control, and slack variables for horizon length $N\in\mathbb{N}$. The time horizon is given by $T\in\mathbb{R}_{>0}$ and the multiple shooting time interval size is given by $h:=\frac{T}{N}$. We denote the start and end points by $\bar{x}_0,\,\bar{x}_N\in\mathbb{R}^{n_x}$. Let $f\colon \mathbb{R}^{n_x}\times\mathbb{R}^{n_u} \to \mathbb{R}^{n_x}$ and $e\colon\mathbb{R}^{n_x}\times\mathbb{R}^{n_u} \to \mathbb{R}^{n_e}$ denote the system dynamics and the stage constraints, respectively. Additionally, let $\mu_0,\,\mu_N\in\mathbb{R}^{n_s}_{> 0}$ denote penalty parameters and let $\mathbb{U}_k,\,\mathbb{X}_k$ denote convex polytopes.

Among other alternatives, Sequential Quadratic Programming (SQP) methods can be used to solve \eqref{OCP}. Due to the nonlinearities introduced in the system dynamics and motion planning constraints in \eqref{OCP} and the difficulty of finding a good initial guess, it is often required to introduce a globalization strategy that ensures global convergence of the iterates. For an overview of different globalization strategies for SQP methods we refer to \cite{Conn2000} and \cite{Nocedal2006}. These standard methods often require tuning of many parameters to ensure global convergence. Moreover, the iterates are in general infeasible, which can lead to inconsistent inputs in a real-time control context where an SQP method may need to be terminated early \cite{Tenny2004}.

In \cite{Wright2004}, the so-called Feasibility-Perturbed Sequential Quadratic Programming (FP-SQP) method was introduced. This method overcomes many of the aforementioned drawbacks by keeping all iterates feasible. This avoids the occurrence of infeasible, physically impossible states and controls. Retaining feasibility is achieved by an additional projection step onto the feasible set. A forward simulation of the system dynamics with the inputs of the SQP step can achieve feasibility with respect to the dynamic constraints. In \cite{Tenny2004}, a more advanced forward simulation is used that also ensures feasibility with respect to linear input and path constraints. In \cite{Bock2007}, feasibility iterations based on repeated evaluation of zero-order information were introduced in the context of real-time Nonlinear Model Predictive Control (NMPC) algorithms under the name "feasibility improvement iterates". In \cite{Zanelli2021}, a feasible SQP method was introduced that uses these feasibility iterations to ensure the feasibility of every SQP iterate. 

Due to the linear objective function in \eqref{OCP}, we propose a Feasible Sequential Linear Programming (FSLP) algorithm that maintains the global convergence properties of FP-SQP. SLP methods were first introduced in \cite{Griffith1961} and in the case that the optimal solution is fully determined by the active set of constraints, they achieve locally quadratic convergence \cite{Messerer2021}. 

The proposed algorithm consists of outer iterations calculating a standard SLP step and inner feasibility iterations projecting the step onto the feasible set. The inner iterations repeatedly solve Linear Programs (LP) and converge linearly towards a feasible point. During this process, only zero-order information of the constraints, i.e., their residuals need to be reevaluated. In contrast to \cite{Zanelli2021}, the convergence of the iterations is achieved by a trust-region. We emphasize that the feasibility iterations do not depend on the optimal control problem structure and are suitable for general NLP. The overall algorithm is free of second-order derivative information and solves only LPs. Therefore, the overall computational cost of the proposed algorithm is low.

The paper is structured as follows. In Section \ref{sec:FSLP}, the outer feasible sequential linear programming algorithm and its convergence behavior are presented. Section \ref{sec:innerFeasIter} discusses the inner feasibility iteration and the main local convergence results are derived. Simulation results on time-optimal motion planning of an overhead crane are presented in Section \ref{sec:numEx}. This paper concludes with Section \ref{sec:conclusion}.

\section{Feasible Sequential Linear Programming}
\label{sec:FSLP}
In this section, we first introduce the problem formulation and the notation. Then, we describe the outer feasible sequential linear programming algorithm and state its local and global convergence properties.
\subsection{Notation \& Preliminaries}
For simplicity of presentation, we bring \eqref{OCP} in a more general form. All state and control variables as well as the step size are stacked into a vector $y\in\mathbb{R}^{n_y}$. The slack variables in the TOCP are stacked in a vector $s_{\mathrm{OCP}}\in\mathbb{R}^{n_s}$. We introduce additional slack variables $s_{\mathrm{NLP}}\in\mathbb{R}^{n_z}$ to bring \eqref{OCP} in the desired problem structure and let $s:=(s_{\mathrm{OCP}},\,s_{\mathrm{NLP}})=[s_{\mathrm{OCP}}^{\top},\,s_{\mathrm{NLP}}^{\top}]^{\top}$. Then, the decision variable $w\in\mathbb{R}^{n_w}$ is defined by $w:=(y,\,s)$. The TOCP \eqref{OCP} can be reformulated in the following way
\begin{equation}
\label{nlp}
\begin{split}
\min_{w \in \mathbb{R}^{n_w}} \: c^{\top}w\quad\mathrm{s.t.} \:\: Cw + g(P_yw)=0,\:\: Aw+b\leq 0.
\end{split}
\end{equation}
Here, let $c\in\mathbb{R}^{n_w}$, $b\in\mathbb{R}^{n_b}$, $A\in\mathbb{R}^{n_b\times n_w}$ with full row rank $n_b$  and let $g\colon\mathbb{R}^{n_y}\to\mathbb{R}^{n_g}$ with $C\in\mathbb{R}^{n_g\times n_w}$. The projection matrix $P_y\in\mathbb{R}^{n_y\times n_w}$ is a sparse matrix, composed of rows with each a single 1-entry, that selects the non-slack variables of $w$. Additionally, we will always choose the minimal slack variables with respect to the objective function in the FSLP algorithm. That is
\begin{align*}
s^*(y) := \underset{s\in\mathbb{R}^{n_s+n_z}}{\mathrm{argmin}} \quad &c_s^{\top}s\\
\mathrm{s.t.}\quad &A_y y + A_s s + b \leq 0,\\
&C_y y +C_s s + g(y) = 0,
\end{align*}
with $A = (A_y\,|\,A_s),\, C = (C_y\,|\,C_s),\,c=(c_y,\,c_s)$. From this follows that there exists a constant $\xi>0$ such that
\begin{align}
\label{eq:equivalece}
\xi^{-1}\Vert w_1 - w_2 \Vert_{2} \leq \Vert	P_y(w_1-w_2)\Vert_{\infty} \leq \xi \Vert w_1 - w_2 \Vert_2
\end{align}
for all $w_1,\,w_2 \in \{w \in \mathbb{R}^{n_w} | y \in \mathbb{R}^{n_y}: w = (y,\,s^*(y))\}$.
We define the measure of infeasibility by
\begin{align}
\label{eq:infeasibilityMeasure}
h(w) := \Vert Cw + g(P_y w) \Vert_{\infty} + \Vert [Aw+b]^+ \Vert_{\infty}
\end{align}
where $[Aw+b]^+ := [\max\{A_iw+b_i,\,0\}]_{i=1}^{n_b}$. Let the Lagrangian be defined by $\mathcal{L}(w,\,\lambda,\,\pi) := c^{\top}w + g(x)^{\top} \lambda + (Aw+b)^{\top} \pi$ with Lagrange multiplier vectors $\lambda\in\mathbb{R}^{n_g}$ and $\pi\in\mathbb{R}^{n_b}$. The KKT-conditions for \eqref{nlp} are given by
\begin{equation}
\begin{split}
&\nabla_w \mathcal{L}(w,\,\lambda,\,\pi) = 0,\: Cw + g(P_yw)=0,\: Aw+b\leq 0\\ 
&\pi\geq 0,\quad\pi_i(A_iw+b_i)= 0,\quad\forall i=1,\ldots,\,n_b.
\end{split}
\end{equation}
The feasible set is denoted by $\mathcal{F} := \{w\in\mathbb{R}^{n_w}\,|\, Cw + g(P_y w)=0,\,Aw+b\leq 0\}$. In Section \ref{sec:numEx}, we make use of a more restrictive definition of feasibility. Let the \textit{zero slack feasible set} with respect to the OCP \eqref{OCP} be $\mathcal{F}_{\mathrm{OCP}}:=\{ w\in\mathbb{R}^{n_w}\, |\,w\in\mathcal{F},\, s_{\mathrm{OCP}}=0\}$.
For a given initial guess $\hat{w}_0\in\mathcal{F}$ we define the level set $L_0(\hat{w}_0)\subset \mathcal{F}$ by $L_0(\hat{w}_0) := \{w\in\mathbb{R}^{n_w}\,|\, w\in\mathcal{F},\,c^{\top}w\leq c^{\top}\hat{w}_0\}$. We denote the closed ball around $\hat{w}$ with radius $\gamma>0$ by $\mathcal{B}(\hat{w},\,\gamma):=\{w\: |\:\Vert\hat{w}-w\Vert_2\leq\gamma \}$. When it is not further specified $\Vert \cdot \Vert$ denotes the Euclidean norm.

\subsection{Description of the algorithm}
Here, we propose the FSLP algorithm which is a special case of the algorithmic framework FP-SQP \cite{Wright2004} that keeps all iterates feasible due to a \textit{feasibility perturbation step}. In \cite{Wright2004}, there is no specific feasibility perturbation technique given. We introduce a novel feasibility perturbation technique in Section \ref{sec:innerFeasIter}. Due to the trust-region used in FP-SQP, there are no strong requirements on the Hessian approximation. Therefore, choosing a zero matrix in the SQP subproblem yields an LP and transforms FP-SQP into an FSLP algorithm while keeping its global convergence theory valid. FSLP is described in Algorithm \ref{alg:FPSQP}.

\begin{figure}[H]
\begin{center}
\resizebox{0.25\textwidth}{!}{
\begin{tikzpicture}
\draw[dashed, -] (3, 0) -- (3, 2.5) node[left] {};
\draw[-{Latex[length=1mm, width=1.5mm]}] (0, 0) -- (3, 0) node[left] {};
\node (C) at (1.5,-0.3) {$\hat{w}$};
\draw[-{Latex[length=1mm, width=1.5mm]}] (0, 0) -- (3, 2.5) node[left] {};
\draw[-{Latex[length=1mm, width=1.5mm]}] (0, 0) -- (2, 2.2) node[left] {};
\node (C) at (0.85,1.5) {$\tilde{w}$};
\draw[dashed, -{Latex[length=1mm, width=1.5mm]}, red] (3, 2.5) -- (2, 2.2) node[right] {};
\node (C) at (2.1,1.25) {$\bar{w}$};
\node (D) at (3.9,-1.0) {$g(P_y w)=0$};
\draw[domain=-pi/8:pi/3, smooth, variable=\t] plot ({3*cos(\t r)}, {3*sin(\t r)});
\filldraw[black] (0,0) circle (2pt) node[anchor=north west]{$O$};
\end{tikzpicture}
}
\end{center}
\vspace{-5mm}
\caption{Visualization of feasibility perturbation.}
\label{fig:three steps}
\vspace{-4mm}
\end{figure}
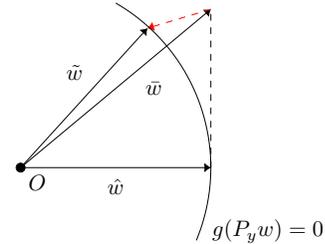

Let $\hat{w}\in\mathcal{F}$ the point of linearization, then the trust-region LP is defined as follows
\begin{subequations}
\label{FP-QP}
\begin{align}
\min_{w \in \mathbb{R}^{n_w}} \quad&c^{\top} w\\
\mathrm{s.t.}\quad &Cw + \nabla g(P_y\hat{w})^{\top}P_y (w-\hat{w}) = 0 \label{FP-QPequal},\\
\quad &Aw +b \leq 0\label{FP-QPinequal},\\
\quad &||P_y(w-\hat{w})||_{\infty} \leq \Delta.
\end{align}
\end{subequations}
After solving \eqref{FP-QP}, its solution $\bar{w}\in\mathbb{R}^{n_w}$ is projected onto the feasible set. The feasibility perturbed iterate is denoted by $\tilde{w}\in\mathbb{R}^{n_w}$. The connection of these three iterates is visualized in Fig. \ref{fig:three steps}. 
For the global convergence theory in \cite{Wright2004} to hold, it is required that $\tilde{w} \in \mathcal{F}$ and that the \textit{projection ratio} is below a certain threshold, i.e., there exists a continuous monotonically increasing function $\phi\colon[0,\,\Delta_{\max,\,2}]\to[0,\,\frac{1}{2}]$ with $\phi(0)=0$ such that 
\begin{align}
\label{eq:asymptoticExactness}
\frac{\Vert\bar{w} -\tilde{w}\Vert_2}{\Vert\bar{w}-\hat{w}\Vert_2} \leq \phi(\Vert\bar{w}-\hat{w}\Vert_2), 
\end{align}
where $\Delta_{\max,\,2}>0$ will be defined in Theorem \ref{theo:asExac}. As globalization strategy, a trust-region approach is used. Since every iterate remains feasible, the merit function is chosen as the objective. The ratio of actual to predicted reduction decides upon step acceptance or rejection. It is defined by
\begin{align}
\label{eq:tr_ratio}
\rho(\hat{w},\,\bar{w},\,\tilde{w}) = \frac{c^{\top}(\hat{w} - \tilde{w})}{c^{\top}(\hat{w}-\bar{w})}.
\end{align}
As termination criterion, we use the linear model of the objective function
\begin{align*}
m(\bar{w};\,\hat{w}) := c^{\top} (\bar{w}-\hat{w}).
\end{align*}
If the model does not decrease through a new LP iterate $\bar{w}$, an optimal point was found and the FSLP algorithm is terminated.
\begin{algorithm}[thpb]
\caption{Feasible Sequential Linear Programming}
\label{alg:FPSQP}
\begin{algorithmic}[1]
\REQUIRE initial point $\hat{w}_0\in\mathcal{F}$, projection matrix $P_y$, initial trust-region radius $\Delta_0 \in (0,\,\tilde{\Delta}]$, trust-region upper bound $\tilde{\Delta}\geq 1$, $\sigma \in (0, 1/4)$, $\sigma_{\mathrm{outer}}\in(0,10^{-5})$, $0<\alpha_1<1<\alpha_2 <\infty$, $0<\eta_1<\eta_2<1$;
\FOR{$k=0,\,1,\,2,\ldots$}
\STATE Obtain $\bar{w}_k$ by solving \eqref{FP-QP}
\IF{$\vert m(\bar{w}_k,\,\hat{w}_k)\vert \leq \sigma_{\mathrm{outer}}$}
\STATE \textbf{STOP}
\ENDIF
\STATE Seek $\tilde{w}_k\in\mathcal{F}$ that fulfills \eqref{eq:asymptoticExactness} with Algorithm \ref{alg:zoiterations}%
\IF{no such $\tilde{w}_k$ is found}
\STATE $\Delta_{k+1} \leftarrow \alpha_1 ||P_y(\bar{w}_k-\hat{w}_k)||_{\infty}$; $\hat{w}_{k+1} \leftarrow \hat{w}_k$
\ELSE
\STATE Calculate $\rho_k$ from \eqref{eq:tr_ratio}
\IF{$\rho_k < \eta_1$}
\STATE $\Delta_{k+1} \leftarrow \alpha_1 ||P_y(\bar{w}_k-\hat{w}_k)||_{\infty}$
\ELSIF{$\rho_k > \eta_2$ and $||P_y(\bar{w}_k-\hat{w}_k)||_{\infty} = \Delta_k$} 
\STATE $\Delta_{k+1} \leftarrow \min(\alpha_2 \Delta_k, \tilde{\Delta})$
\ELSE
\STATE $\Delta_{k+1} \leftarrow \Delta_k$
\ENDIF
\IF{$\rho_k > \sigma$}
\STATE $\hat{w}_{k+1} \leftarrow \tilde{w}_k$
\ELSE
\STATE $\hat{w}_{k+1} \leftarrow \hat{w}_k$
\ENDIF
\ENDIF
\ENDFOR
\end{algorithmic}
\end{algorithm}


\subsection{Global convergence of FP-SQP and FSLP}
The global convergence theory of FP-SQP can directly be applied to FSLP. In order to prove global convergence the following assumptions are made in \cite{Wright2004}:
\begin{assumption}
\label{ass:boundedL0}
For a given $\hat{w}_0$, the level set $L_0(\hat{w}_0)$ is bounded, and the function $g$ in \eqref{nlp} is twice continuously differentiable in an open neighborhood of $ L_0(\hat{w}_0)$.
\end{assumption}
\begin{assumption}
\label{ass:boundedDistance}
For every point $\hat{w}\in L_0(\hat{w}_0)$, there are constants $\zeta,\,\hat{\Delta}>0$ such that for all $w\in \mathcal{B}(\hat{w},\,\xi \hat{\Delta})$ we have 
\begin{align*}
\min_{v \in\mathcal{F}}\, \Vert v-w\Vert_2 \leq \zeta (\Vert g(w)\Vert_2 + \Vert [Aw+b]^+\Vert_2),
\end{align*}
where $\xi$ was defined in \eqref{eq:equivalece}.
\end{assumption}

The last assumption requires that the distance of $\bar{w}$ to the closest feasible point is bounded by the constraint violation of $\bar{w}$. This ensures the existence of a feasible point $\tilde{w}$ satisfying \eqref{eq:asymptoticExactness} \cite{Wright2004}. 

Since FSLP is a special case of FP-SQP we apply the following global convergence result:
\begin{theorem}[Global convergence of FP-SQP \cite{Wright2004}]
Suppose that Assumptions \ref{ass:boundedL0}, \ref{ass:boundedDistance} hold then all limit points of Algorithm \ref{alg:FPSQP} either are KKT points or else fail to satisfy the Mangasarian-Fromowitz constraint qualification.
\end{theorem}

\subsection{Local convergence of FSLP}
In the following, we show that in the case where the solution of an NLP is fully determined by the active constraints, following \cite{Kim2016} and \cite{Messerer2021}, we obtain local quadratic convergence if the iterates are not projected on the feasible set. In practice, local quadratic convergence is often obtained for projected iterates as we show in section \ref{sec:numEx}. In the fully determined case, solving the NLP \eqref{nlp} is equivalent to finding a feasible point to the active constraints. Applying the classical Newton method to this root-finding problem yields quadratic convergence \cite{Messerer2021}. If an LP has at least one solution, then at least one solution lies in a vertex of the feasible set. In the fully determined case, the NLP solution is locally unique and lies in a vertex of the feasible set. Therefore, the unprojected iterates of FSLP converge quadratically. We adopt the convergence result of \cite{Messerer2021} to our algorithm.
\begin{corollary}[Local quadratic convergence]
Assume $z^*=(\hat{w}^*,\,\lambda^*,\pi^*) = (\hat{w}^*,\,\vartheta^*)$ is a KKT point of \eqref{nlp}, at which LICQ and strict complementarity hold. If $\hat{w}^*$ is fully determined by the active constraints, then $\hat{w}^*$ is a local minimizer. Additionally, there exists a trust-region radius $\tilde{\Delta}>0$ and a $\bar{k}\in\mathbb{N}$ such that for all $\hat{w}_k$ with $\Vert P_w(\hat{w}_k-\hat{w}^*)\Vert\leq \tilde{\Delta}$ and $k>\bar{k}$, then FSLP converges Q-quadratically in the primal variable $\hat{w}$, and R-quadratically in the dual variable $\vartheta$, i.e., there are constants $c_1,\,c_2\in\mathbb{R}_+$ such that
\begin{align*}
&\Vert \hat{w}_{k+1}-\hat{w}^*\Vert \leq c_1\Vert \hat{w}_k-\hat{w}^*\Vert^2\\
\mathrm{while}\quad&\Vert\vartheta_{k+1}-\vartheta^*\Vert \leq c_2\Vert \hat{w}_k-\hat{w}^*\Vert.
\end{align*}
\end{corollary}
\begin{proof}
Due to the trust-region radius $\tilde{\Delta}$ all iterates $\hat{w}_k$ for $k>\bar{k}$ are in the region of local convergence around $\hat{w}^*$. Then, the proof follows from \cite{Messerer2021}.
\end{proof}

\section{Inner Feasibility Iterations}
\label{sec:innerFeasIter}
In this section, we present the inner feasibility iterations to obtain a feasible step $\tilde{w}\in\mathcal{F}$ from the LP step $\bar{w}$. This is the main algorithmic contribution of this paper. After introducing the algorithm, we propose a termination heuristic for efficient implementation and we prove local convergence of the inner iterates towards a feasible point. In particular, we show that this step fulfills the projection ratio condition \eqref{eq:asymptoticExactness}.
\subsection{Description of the algorithm}
The feasibility iterations are closely related to second-order corrections which were first introduced in \cite{Fletcher1982} in order to avoid slow convergence in globalized SQP methods. In fact, if the algorithm is started at $\bar{w}$ the very first feasibility iteration coincides with a second-order correction. Further iterations perform higher-order corrections of the constraints.
Our approach is very similar to \cite{Zanelli2021}, but it differs in the use of a trust-region as a means to impose local convergence in the inner iterations and global convergence in the outer iterations. Let $\hat{w}\in\mathcal{F}$ be the outer iterate and $w_l\in\mathbb{R}^{n_w}$ be an inner iterate for a given inner iterate counter $l\in\mathbb{N}$. We fix the Jacobian of $g$ at $P_y\hat{w}$, i.e., $G^{\top} := \nabla g(P_y\hat{w})^{\top}P_y$, and define
\begin{align}
\delta(w_l,\,\hat{w}) := g(P_y w_l)-g(P_y \hat{w}) - G^{\top}(w_l-\hat{w}).
\end{align}
Let $\delta_l = \delta(w_l,\,\hat{w})$, then we define the parametric linear program $\mathrm{PLP}(\delta_l;\,\hat{w},\Delta)$ as
\begin{subequations}
\label{PLP}
\begin{align}
\min_{w \in \mathbb{R}^{n_w}} \quad&c^{\top}w\\
\mathrm{s.t.}\quad &\delta_l +Cw + G^{\top} (w-\hat{w}) = 0,\label{eq:ZOnleq}\\
\quad &Aw +b \leq 0,\\
\quad &||P_y(w-\hat{w})||_{\infty} \leq \Delta.
\end{align}
\end{subequations}
Its solution will be denoted by $w_{\mathrm{PLP}}^*(\delta_l;\,\hat{w},\Delta)$. We note that for $w_l \leftarrow \hat{w}$, we obtain LP \eqref{FP-QP} and for $w_l \leftarrow \bar{w}$, we obtain a standard second-order correction problem as defined in \cite{Conn2000} or \cite{Nocedal2006}. The algorithm is based on the iterative solution of PLPs. In every iteration the nonlinear constraints are re-evaluated at $w_l$ in the term $\delta_l$ and another PLP is solved. The solution of this problem is denoted by $w_{l+1}$ and the procedure is repeated. We will show that for $l\to\infty$ the limit of the sequence $\{w_l\}_{l\in\mathbb{N}}$ is $\tilde{w}\in\mathcal{F}$. Algorithm \ref{alg:zoiterations} presents this feasibility improvement strategy in detail.
\begin{algorithm}[thpb]
\caption{Inner Feasibility Iterations}
\label{alg:zoiterations}
\begin{algorithmic}[1]
\REQUIRE $\hat{w}\in\mathcal{F}$, fixed Jacobian $G^{\top}=\nabla g(\hat{w})^{\top}P_y$. Let $\bar{w}$ be the solution of \eqref{FP-QP} at $\hat{w}$. $n_{\mathrm{watch}}\in\mathbb{N}$, $\kappa_{\mathrm{watch}}<1$, $\sigma_{\mathrm{inner}}\in (0,\,10^{-5})$;
\ENSURE $\tilde{w}$
\STATE $w_0\leftarrow \bar{w}$
\FOR{$l=0,\,1,\,2, \ldots$}
\IF{$h(w_l) \leq \sigma_{\mathrm{inner}}$ and $\Vert \bar{w}-w_l\Vert / \Vert \bar{w}-\hat{w}\Vert < 1/2$}
\STATE $\tilde{w}\leftarrow w_l$
\STATE \textbf{STOP}
\ENDIF
\STATE Solve $\mathrm{PLP}(\delta_l,\,\hat{w},\,\Delta)$
\IF{iterates $w_l$ are diverging according to subsection \ref{sec:TerminationHeuristic}}
\STATE \textbf{STOP}
\ENDIF
\STATE $w_{l+1}\leftarrow w_{\mathrm{PLP}}^*(\delta_l;\,\hat{w},\Delta)$
\ENDFOR
\end{algorithmic}
\end{algorithm}
The feasibility iterations are repeated until convergence towards a feasible point of \eqref{nlp} is achieved. If the iterates are diverging the algorithm is terminated, the inner algorithm returns to the outer algorithm Algorithm \ref{alg:FPSQP}, and the trust-region radius is decreased. This strategy is presented in detail in the following subsection \ref{sec:TerminationHeuristic}.

Particularly, in every iteration of Algorithm \ref{alg:zoiterations}, the constraints are re-evaluated. This is advantageous in applications where the evaluation of first- and second-order derivatives is expensive. Assuming that the solutions of the PLPs do not differ much, the solution of the feasibility problem can be computed quite cheaply by using an active set solver \cite{Zanelli2021}.

The subsequent sections focus on the termination heuristic for Algorithm \ref{alg:zoiterations} and on its local convergence behavior. We state that the limit of the feasibility iterations is indeed a feasible point. Afterwards, we prove local convergence and the satisfaction of the projection ratio condition \eqref{eq:asymptoticExactness}.

\subsection{Termination heuristic}
\label{sec:TerminationHeuristic}
For an efficient FSLP algorithm, it is crucial to stop the inner feasibility iterations early if the iterates are not converging. We recall that the optimal solution must be feasible, i.e., $w^*\in\mathcal{F}$ and that the projection ratio condition \eqref{eq:asymptoticExactness} must be satisfied. The termination heuristic consists of the following steps.

As shown in Section \ref{sec:numEx}, the projection ratio does not change much for converging feasibility iterations. Therefore, we observe the projection ratio for all inner iterates, i.e., $\Vert \bar{w} - w_l\Vert/\Vert\bar{w}-\hat{w}\Vert$. If it is higher than $1.0$, the inner iterations are aborted and the trust-region radius is decreased.

Additionally, we estimate the contraction rate $\kappa$ of the algorithm with the following formula
\begin{align}
\kappa_l = \frac{\Vert w_{l+1}-w_l\Vert}{\Vert w_l-w_{l-1}\Vert}.
\end{align}

The convergence of the inner iterations is observed through a watchdog strategy. After $n_{\mathrm{watch}}$ inner iterations, the contraction rate for these $n_{\mathrm{watch}}$ steps is checked, if it is not below a threshold $\kappa_{\mathrm{watch}}$ the iterations are aborted. The projection ratio is also checked. If it is not below $0.5$, the algorithm also terminates.

If the feasibility measure $h(w_l)$ is below a feasibility tolerance $\sigma_{\mathrm{inner}}$, the algorithm terminates. If the maximum number of iterations $n_{\mathrm{max}}$ are reached without achieving convergence, the inner iterations are aborted and the trust-region radius is decreased.

\subsection{Limit of feasibility iterations}
For completeness of presentation, we state a result about the limit of the feasibility iterations, first derived in \cite{Bock2007}:
\begin{lemma}[Limit of feasibility improvement]
Assume that for fixed $\hat{w}$, the sequence of feasibility iterates $\{w_l\}_{l\in\mathbb{N}}$ converges towards a $w^*$, and let $\pi^*$ be the corresponding Lagrange multipliers of \eqref{eq:ZOnleq} in $w^*$. Then $w^*$ and $\pi^*$ belong to a KKT point of the problem
\begin{subequations}
\label{ZOP}
\begin{align}
\min_{w\in\mathbb{R}^{n_w}}\quad &(c + (G-P_y^{\top} \nabla g(P_y w^*))\pi^*)^{\top} w\\
\mathrm{s.t.}\quad&Cw + g(P_y w) = 0,\quad Aw +b \leq 0,\\
& \Vert P_y(w-\hat{w})\Vert_{\infty}\leq \Delta.
\end{align}
\end{subequations}
\end{lemma}
\begin{proof}
The proof follows by comparison of the KKT-conditions of \eqref{PLP} and \eqref{ZOP} in the same fashion as in \cite{Bock2007}.
\end{proof}
\subsection{Local convergence of feasibility iterations}
\label{subsec:localconvergence}
Next, we provide a proof of the local convergence of the feasibility iterations. In order to prove local contraction, we make use of the following assumption which is in the context of generalized equations referred to as strong regularity \cite{Robinson1980}.
\begin{assumption}[Strong regularity]
\label{as:invertiblePLP}
For all $\hat{w}\in\mathcal{F}$ exist $L_1,\, L_2,\,\bar{\Delta}>0$ such that for all $\Delta\leq\bar{\Delta}$ and for all $\delta_1,\,\delta_2\in \mathcal{B}(0,\,L_2\,\Delta)$ it holds that
$
\Vert w_{\mathrm{PLP}}^*(\delta_1;\,\hat{w},\Delta)-w_{\mathrm{PLP}}^*(\delta_2;\,\hat{w},\Delta)\Vert \leq L_1 \Vert \delta_1-\delta_2\Vert.
$
\end{assumption}

We note that $w_{l+1}=w_{\mathrm{PLP}}^*(\delta_l;\,\hat{w},\Delta)$.
Then, we can state the local contraction result:
\begin{theorem}[Local linear convergence proportional to $\Delta$]
\label{theo:localContractionJosephyNewton}
Let $\hat{w}\in\mathcal{F}$ and let Assumption \ref{as:invertiblePLP} hold, then there exist $\Delta_{\mathrm{max},\,1},\,L>0$ such that for all $\Delta\leq\Delta_{\mathrm{max},\,1}$ the iterates $\{w_l\}_{l\in\mathbb{N}}$ of Algorithm \ref{alg:zoiterations} converge towards a point $w^*\in\mathcal{F}$ and the contraction rate is proportional to $\Delta$, i.e.,
$
\Vert w_{l+1}-w^*\Vert \leq L \Delta \Vert w_{l}-w^*\Vert.
$
\end{theorem}
\begin{proof}
Let $\Delta \leq \bar{\Delta}$. Using the definition of the solution of $\mathrm{PLP}(\delta;\,\hat{w},\Delta)$ and Assumption \ref{as:invertiblePLP} yields
$
\Vert w_{l+1} - w^* \Vert = \Vert w_{\mathrm{PLP}}^*(\delta_{l};\,\hat{w},\Delta) - w_{\mathrm{PLP}}^*(\delta^*;\,\hat{w},\Delta) \Vert
\leq L_1 \Vert \delta_l - \delta^*\Vert.
$
For simplicity of notation we define $\tilde{G} \colon w \mapsto P_y\nabla g(P_y w)$ and observe $\tilde{G}(\hat{w})=G$. We note that by applying the fundamental theorem of calculus and using Lipschitz continuity of $\tilde{G}$ with Lipschitz constant $\tilde{L}>0$ we obtain
\begin{align*}
&\Vert\delta_l - \delta^*\Vert = \Vert g(P_y w_l)-g(P_y w^*)-G^{\top}(w_l-w^*)\Vert\\
&\leq \int_0^1 \Vert(\tilde{G}(w^* + t(w_l-w^*))-G)^{\top}\Vert\:\Vert w_l-w^*\Vert \mathrm{d}t\\
&\leq \int_0^1 \tilde{L} \Vert w^* + t(w_l-w^*) -\hat{w}\Vert \mathrm{d}t\:\Vert w_l-w^*\Vert\\
&\leq \tilde{L}\,\Delta\,\Vert w_l-w^*\Vert.
\end{align*}
Defining $L=L_1\, \tilde{L}$ yields
\begin{align*}
\Vert w_{l+1} - w^* \Vert \leq L \Delta\,  \Vert w_{l}-w^*\Vert.
\end{align*}
From this follows $\Delta_{\mathrm{max},\,1}<\frac{1}{L}$.
\end{proof}
\subsection{Projection ratio}
\label{subsec:asExac}
In this section, we provide a proof that condition \eqref{eq:asymptoticExactness} about the projection ratio holds for the optimal solution of the inner feasibility iterations. We note that $w_0=\hat{w}$, $w_1=\bar{w}$, and $w_l\to\tilde{w}$, i.e., $\hat{w}=w^*$. For an illustration of the three different iterates, we refer to Fig. \ref{fig:three steps}.
\begin{theorem}[Projection ratio]
\label{theo:asExac}
Let Assumption \ref{as:invertiblePLP} hold. Moreover, let $\tilde{w}$ be the limit of Algorithm \ref{alg:zoiterations}, $\bar{w}$ be the solution of \eqref{FP-QP}, and let $\hat{w}$ be the point of linearization, then there exists $\Delta_{\mathrm{max},\,2}$ such that a function $\phi\colon\mathbb{R}\to[0,\,1/2]$ satisfying condition \eqref{eq:asymptoticExactness} exists.
\end{theorem}
\begin{proof}
We interpret the solution of \eqref{FP-QP} $\bar{w}$ as the first iterate of Algorithm \ref{alg:zoiterations}, i.e., $w_0=\hat{w}$ and $w_1=\bar{w}$. We also note that $\tilde{w}=w^*$. In the same way, as done in Theorem \ref{theo:localContractionJosephyNewton}, we obtain the following estimate
\begin{align*}
\Vert\delta_{l+1} &-\delta_{l}\Vert\\
&\leq \int_0^1 \tilde{L} \Vert w_l + t(w_{l+1}-w_l) -\hat{w}\Vert \mathrm{d}t\:\Vert w_{l+1}-w_l\Vert\\
&\leq \tilde{L} \int_0^1 \Vert w_l -w_0\Vert + t\Vert w_{l+1}-w_l \Vert \mathrm{d}t\:\Vert w_{l+1}-w_l\Vert\\
&\leq \tilde{L} ( \Vert w_l -w_0\Vert + \frac{1}{2}\Vert w_{l+1}-w_l \Vert )\:\Vert w_{l+1}-w_l\Vert.
\end{align*}
From this follows
\begin{align*}
\Vert w_{l+2} &- w_{l+1}\Vert\\
&\leq L ( \Vert w_l -w_0\Vert + \frac{1}{2}\Vert w_{l+1}-w_l \Vert )\:\Vert w_{l+1}-w_l\Vert.
\end{align*}
For simplicity, we define
\begin{align*}
\kappa_l := L (\Vert w_l -w_0\Vert + \frac{1}{2}\Vert w_{l+1}-w_l \Vert),
\end{align*}
which gives
\begin{align*}
\Vert w_{l+2} - w_{l+1}\Vert \leq \kappa_l\:\Vert w_{l+1}-w_l\Vert.
\end{align*}
We also define $\kappa_{\max}:= 2 L \Delta.$
For the case $l=0$ we get $\kappa_0 = \frac{1}{2}L \Vert w_1-w_0\Vert$ such that
\begin{align*}
\Vert w_2 - w_1\Vert \leq \tau \Vert w_1 - w_0\Vert^2
\end{align*}
where $\tau:=\frac{1}{2}L$. We proceed by a telescoping sum argument. It holds
\begin{align*}
\sum_{l=0}^{K}w_{l+2}-w_{l+1} = w_{K+2}-w_1
\end{align*}
and
\begin{align*}
\Vert w_{l+2}-w_{l+1}\Vert \leq (\kappa_{\max})^l\Vert w_{2}-w_1 \Vert.
\end{align*}
Then,
\begin{align*}
\Vert w_{K+2}-w_1 \Vert &\leq \sum_{l=0}^{K}\Vert w_{l+2}-w_{l+1}\Vert\\
&\leq \sum_{l=0}^{K}(\kappa_{\max})^l\Vert w_{2}-w_1 \Vert\\
&\leq\left (\sum_{l=0}^{K}(\kappa_{\max})^l\right) \tau\Vert w_{1}-w_0\Vert^2\\
&\leq \frac{\tau}{1-\kappa_{\max}} \Vert w_{1}-w_0\Vert^2.
\end{align*}
For $K\to\infty$ and with the original notation of \eqref{eq:asymptoticExactness}, we get
\begin{align*}
\Vert \tilde{w}-\bar{w} \Vert \leq\frac{\tau}{1-\kappa_{\max}}\Vert \bar{w}-\hat{w}\Vert^2.
\end{align*}
From this we define $\phi$ as 
\begin{align*}
\phi(\Vert \bar{w}-\hat{w}\Vert) = \frac{0.5 L \Vert \bar{w}-\hat{w}\Vert}{1-2L^2 \Delta},
\end{align*}
where $\phi$ is always smaller than $\frac{0.5 L \Delta}{1-2L^2 \Delta}$. If we choose $\Delta_{\mathrm{max},\,2} \leq \frac{1}{(1+2L)L}$, then $\phi(\cdot)\leq 1/2$.
\end{proof}

\section{Numerical Example}
\label{sec:numEx}
In this section, we present simulation results for the FSLP algorithm. We first demonstrate the quadratic convergence behavior of FSLP on an illustrative example of a fully determined system. Subsequently, we test FSLP on a time-optimal point-to-point motion problem of an overhead crane. We show, that the theoretically derived contraction properties of the inner feasibility iterations can be empirically verified on the test example and that the desired projection ratio property holds. Finally, we compare FSLP with the state-of-the-art NLP solver \texttt{Ipopt} \cite{Waechter2006} on 100 different problem instances. \texttt{Ipopt} uses the linear solver \texttt{ma57} from the HSL library \cite{HSL}.

\subsection{Implementation}
We use the Python interface of the open source software \texttt{CasADi} \cite{Andersson2019} to model the optimization problem and for a prototypical implementation of the FSLP solver. As LP solver, we use the dual simplex algorithm of \texttt{CPLEX} version 12.8 \cite{cplex2017v12} which can be called from \texttt{CasADi}. The implementation is open-source and can be found at \url{https://github.com/david0oo/fslp}. As parameters in Algorithm \ref{alg:FPSQP} we chose: $\Delta_0=1,\,\alpha_1=\eta_1=0.25,\,\alpha_2=2,\,\eta_2=0.75,\,\sigma=10^{-8},\,\sigma_{\mathrm{outer}}=10^{-8}$. In Algorithm \ref{alg:zoiterations}, we chose $n_{\mathrm{watch}}=5,\,\kappa_{\mathrm{watch}}=0.3$ and $\sigma_{\mathrm{inner}}=10^{-7}$. The simulations were carried out on a Intel Core i7-10810U CPU.

\subsection{Illustrative example of quadratic convergence}
In this subsection, we illustrate the quadratic convergence of the outer iterations of FSLP in the case of a fully determined system. Therefore, we define the following parametric optimization problem
\begin{align}
\label{eq:exampleFullyDetermined}
\min_{w\in\mathbb{R}^2} w_2\quad\mathrm{s.t.}\quad w_2\geq w_1^2,\,w_2 \geq 0.1 w_1+\varepsilon
\end{align}
with $\varepsilon\in\mathbb{R}$.
It is easy to see that the optimal solution of \eqref{eq:exampleFullyDetermined} for $\varepsilon=-0.06$ is $w^*=(0,\,0)$ and for $\varepsilon=0.06$, it is $w^*=(-0.2,\,0.04)$. For $\varepsilon=0.06$, the solution is fully determined by both constraints and lies in a vertex of the linearized constraints at the solution. For $\varepsilon=-0.06$, the solution is not fully determined and does not lie in a vertex of the linearized constraints at the solution.
\begin{figure}[thpb]
\includegraphics[width=0.48\textwidth]{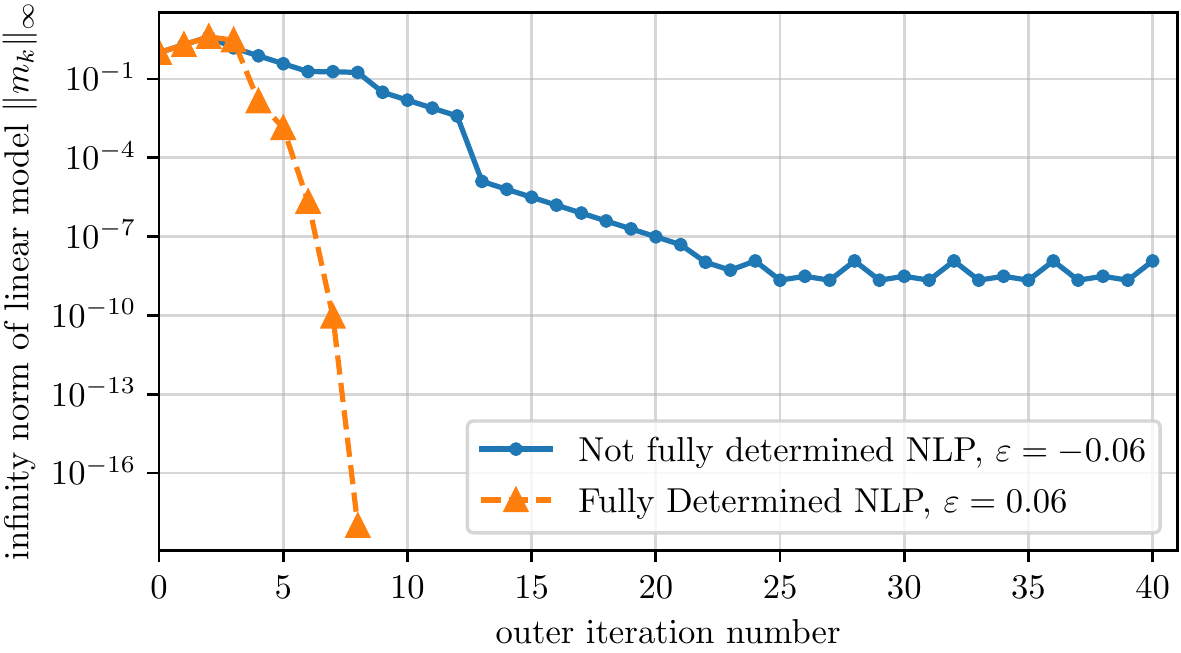}
\vspace{-3mm}
\caption{Comparison of the convergence of FSLP on a fully determined NLP and a not fully determined NLP.}
\label{fig:illustrative_example}
\end{figure}
In Fig. \ref{fig:illustrative_example}, we see that in the case of the fully determined NLP, Algorithm \ref{alg:FPSQP} converges quadratically towards the optimal solution. In the case of the not fully determined problem, the algorithm converges linearly towards the optimal solution, but at a certain accuracy, it is not possible to improve accuracy since all iterates of FSLP lie in a vertex of the boundary of the trust-region. If a step is accepted, the trust-region is increased which results in an overshoot of the steps such that the accuracy decreases again. The convergence results are shown in Fig. \ref{fig:illustrative_example}. The FSLP algorithm was initialized at $\hat{w}_0=(2,\,10)$ for both problems.

\subsection{Point-to-point motion of an overhead crane}
\label{subsec:overheadCrane}
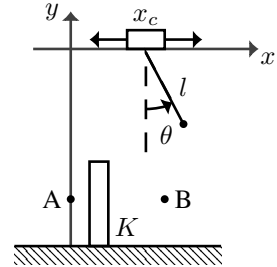
\begin{figure}[thpb]
\begin{center}

\resizebox{0.25\textwidth}{!}{
\begin{tikzpicture}
\draw[-] (-0.3, 0.25) -- (0.8, 0.25) node[left] {};
\draw [draw=none, pattern=north west lines] (-0.3, 0.25) rectangle (0.8, -0.1+0.25);

\draw [] (0.1, 0.25) rectangle (0.2, 1.4-0.7);
\node (A) at (0.3, 0.35) {\scalebox{.4}{$K$}};

\draw [] (0.3, 2-0.7) rectangle (0.5, 2.1-0.7);
\node (A) at (0.4, 1.47) {\scalebox{.4}{$x_c$}};
\draw[-{Latex[length=0.5mm, width=1mm]}] (0.5, 2.1-0.75) -- (0.7, 2.1-0.75) node[left] {};
\draw[{Latex[length=0.5mm, width=1mm]}-] (0.1, 2.1-0.75) -- (0.3, 2.1-0.75) node[left] {};
\draw[-] (0.4, 2-0.72) -- (0.6, 0.9) node[right] {};
\filldraw (0.6,0.9) circle (0.4pt) node[right] {};
\node (B) at (0.6, 1.1) {\scalebox{.4}{$l$}};

\draw[densely dashed]
    (0.4, 0.75) coordinate (a) node[right] {}
    -- (0.4, 2-0.72) coordinate (b) node[left] {}
    -- (0.55,1) coordinate (c) node[above right] {}
    pic["\scalebox{.4}{$\theta$}", draw=black, -{Latex[length=0.5mm, width=1mm]}, angle eccentricity=1.5, angle radius=3mm]
    {angle=a--b--c};
%


\begin{scope}[transparency group, opacity=0.75]
\draw[-{Latex[length=0.5mm, width=1mm]}] (0, 0.25) -- (0, 2-0.7+0.2) node[left] {};
\draw[-{Latex[length=0.5mm, width=1mm]}] (-0.2, 2-0.7) -- (0.6+0.2+0.2, 2-0.7) node[left] {};
\end{scope}
\node (A) at (-0.1, 2-0.7+0.2) {\scalebox{.4}{$y$}};
\node (A) at (0.6+0.2+0.05+0.2, 2-0.7-0.05) {\scalebox{.4}{$x$}};

\filldraw (0,1.2-0.7) circle (0.4pt) node[left] {};
\node (C) at (-0.1, 0.5) {\scalebox{.4}{A}};
\filldraw (0.5,1.2-0.7) circle (0.4pt) node[right] {};
\node (D) at (0.6, 0.5) {\scalebox{.4}{B}};

%

\end{tikzpicture}
}
\end{center}
\vspace{-5mm}
\caption{Schematic illustration of the overhead crane.}
\label{fig:crane}
\vspace{-3mm}
\end{figure}
As main test example, we present the time-optimal point-to-point motion of an overhead crane as shown in Fig. \ref{fig:crane}. The crane can move its position $x_c\, [\mathrm{m}]$ on a rail. The length of the crane hoist is denoted by $l\, [\mathrm{m}]$ and its angle with respect to the rail by $\theta\, [\mathrm{rad}]$. The payload of the crane has the position $p = (x_c + l\sin(\theta), -l\cos(\theta))$. The control inputs are the acceleration of the cart $\ddot{x}_c\, [\mathrm{m/s^2}]$ and the winding acceleration $\ddot{l}\, [\mathrm{m/s^2}]$ of the hoist. Our goal is to move the payload from point of rest A to point of rest B in minimal time $t\, [s]$. The system dynamics are given by
\begin{align*}
l\ddot{\theta} = \cos(\theta)\ddot{x}_c - 2\dot{l}\dot{\theta} -g\sin(\theta),
\end{align*}
where $g=9.81 \mathrm{kg\,m/s^2}$ is the gravitational acceleration. The rectangular obstacle $K\subset\mathbb{R}^2$ in Fig. \ref{fig:crane} is represented by its vertices $v_1,\,v_2,\,v_3,\,v_4\in\mathbb{R}^2$. In order to model obstacle avoidance, we apply the separating hyperplane theorem \cite{Boyd2004}. We introduce the radius of the load $r_{\mathrm{load}}\,[\mathrm{m}]$ and for every time instant $k=1,\ldots,\,N$, we introduce hyperplane variables $u^{\mathrm{h}}_k\in\mathbb{R}^2$, $u^{\mathrm{c}}_k\in\mathbb{R}$ that separate the payload from the obstacle. The constraints are of the following form:
\begin{align*}
&p_k^{\top}u^{\mathrm{h}}_k - u^{\mathrm{c}}_k \leq -r_{\mathrm{load}},\\
&v_i^{\top}u^{\mathrm{h}}_k - u^{\mathrm{c}}_k\geq 0\quad\forall\: i=1,\ldots,\,4,\\
&\Vert u_k^h \Vert_{\infty}\leq 1,\quad \Vert u_k^c \Vert_{\infty}\leq 1
\end{align*}
for all $\quad k=1,\ldots,\,N$.
Algorithm \ref{alg:FPSQP} needs to be initialized with a feasible trajectory. Since it is difficult to provide an initial guess satisfying the start and end position constraints, we have to relax both conditions, i.e, 
\begin{align*}
-s_0 \leq x_0 - \bar{x}_0 \leq s_0,\quad -s_f \leq x_N - \bar{x}_f \leq s_f,
\end{align*}
with $s_0,\,s_f\in\mathbb{R}^{n_x}_{\geq 0}$. The slack variables are penalized by the factor $10^5$ in the objective. Furthermore, $N=20$ and $r_{\mathrm{load}} =0.08\,\mathrm{m}$. Table \ref{tab:parameters} shows the parameters for the box constraints on the variables. The system dynamics were integrated with a Runge Kutta 4 scheme with 20 internal steps per multiple shooting interval. Altogether, we obtain a TOCP formulation as in \eqref{OCP}.

In the simulations, we initialize with the trajectory of iteration zero as in Fig. \ref{fig:opt_trajectories}. These can be created by initializing with initial time $2.5\,\mathrm{s}$, i.e., $h=0.125\,\mathrm{s}$ and forward simulating the system dynamics with the constant control $(0\,\mathrm{m/s^2},\,0.1\,\mathrm{m/s^2})$ from the initial payload position $(0\,\mathrm{m},\,-0.6\,\mathrm{m})$. The trajectories of some iterates of FSLP are shown in Fig. \ref{fig:opt_trajectories}. FSLP terminates after eleven iterations at the optimal solution. After three iterations, the starting condition is already satisfied, after six iterations the start and end conditions are satisfied, i.e., a zero slack feasible solution is found, and after nine iterations, the suboptimal iterate and the optimal solution are barely distinguishable. In the following, we always mean a zero slack feasible solution, if a feasible solution is mentioned.

\begin{figure}[thpb]
\includegraphics[width=0.485\textwidth]{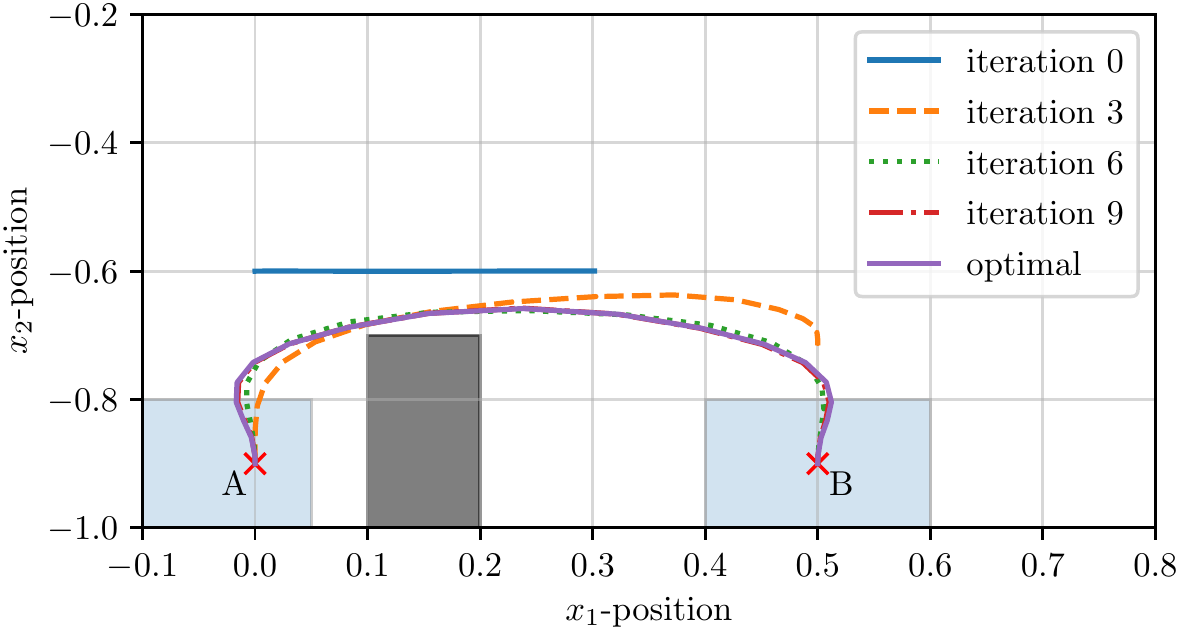}
\vspace{-5mm} 
\caption{Payload trajectories of the iterates of FSLP on the overhead crane problem.}
\label{fig:opt_trajectories}
\vspace{-3mm} 
\end{figure}

\begin{table}[thpb]
\caption{Boundaries of variables.}
\label{tab:parameters}
\centering
\begin{tabular}{|c|c|c|}
\hline
Lower Bound & Description & Upper Bound\\
\hline
$-0.1\,\mathrm{m}$ & $x_c$ & $0.6\,\mathrm{m}$\\
$-0.4\,\mathrm{m/s}$ & $\dot{x}_c$ & $0.4\,\mathrm{m/s}$\\
$10^{-2}\,\mathrm{m}$ & $l$ & $2\,\mathrm{m}$\\
$-0.25\,\mathrm{m/s}$ & $\dot{l}$ & $0.25\,\mathrm{m/s}$\\
$-0.75\,\mathrm{rad}$ & $\theta$ & $0.75\,\mathrm{rad}$\\
$-5\,\mathrm{m/s^2}$ & $\ddot{x}_c$ & $5\,\mathrm{m/s^2}$\\
$-5\,\mathrm{m/s^2}$ & $\ddot{l}$ & $5\,\mathrm{m/s^2}$\\
$0$ & $s_0,\,s_N$ & $+\infty$\\
\hline
\end{tabular}
\vspace{-2mm}
\end{table}

\subsubsection{Local contraction and projection ratio}
At first, we investigate the local convergence behavior of the inner iterations for different trust-region radii. As an example problem, we use the first outer iteration of FSLP initialized with the trajectory of iteration zero as shown in Fig. \ref{fig:opt_trajectories}. In Fig. \ref{fig:localConvergence}, we observe the contraction and projection ratio for different trust-region radii. We see that the iterates converge locally and fulfill \eqref{eq:asymptoticExactness} for trust-region radii below a certain threshold which confirms the theoretical results of subsections \ref{subsec:localconvergence} and \ref{subsec:asExac}.

\begin{figure}[thpb]
\includegraphics[width=0.485\textwidth]{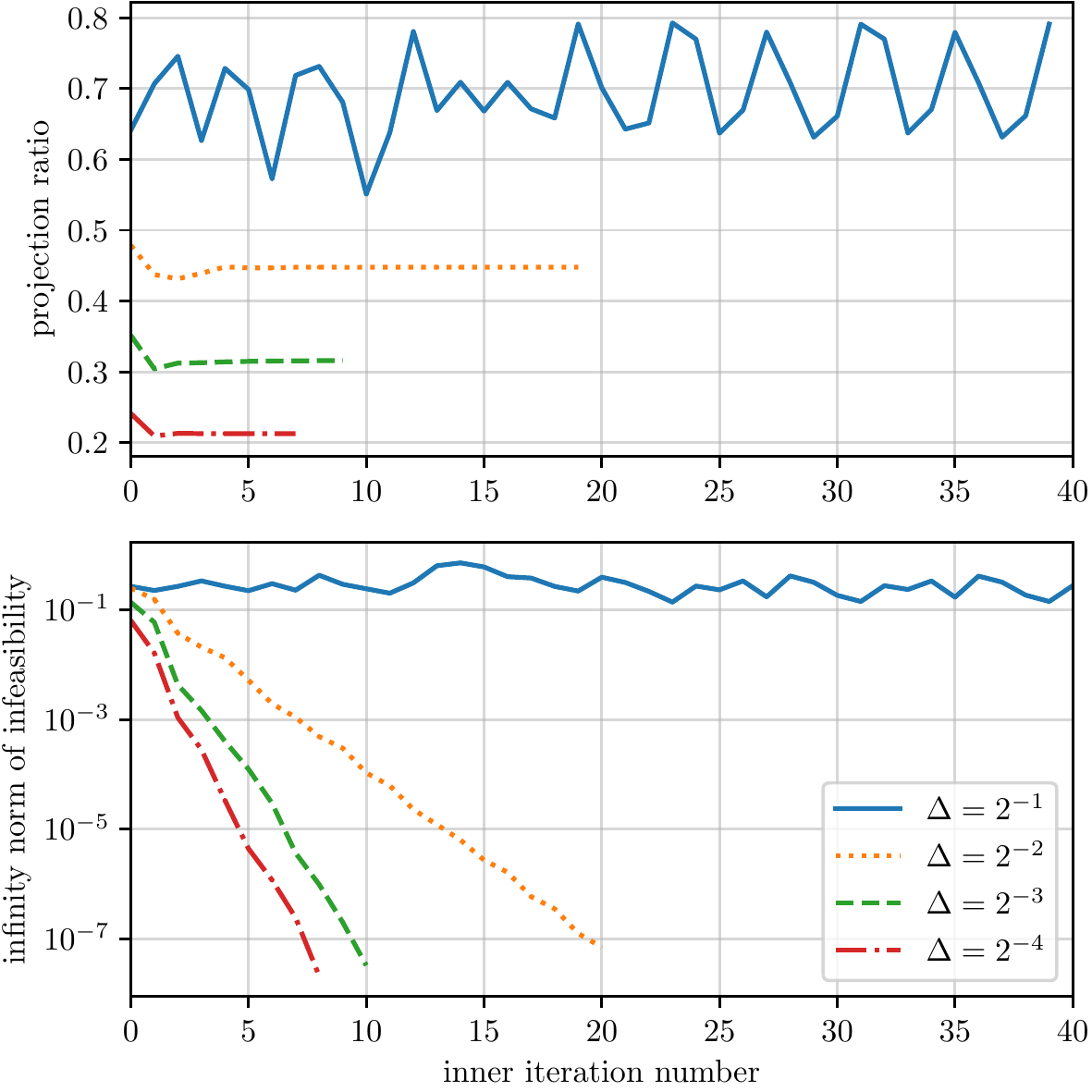}
\vspace{-5mm}
\caption{Comparison of local convergence and projection ratio of feasibility iterations for different trust-region radii.  For all trust-region radii below a certain value between $2^{-1}$ and $2^{-2}$ we obtain linear convergence and satisfaction of condition \eqref{eq:asymptoticExactness}.}
\label{fig:localConvergence}
\vspace{-4mm}
\end{figure}
\subsubsection{Comparison to Ipopt}
In the next experiment, we take ten random perturbations of starting position A and ten perturbations of the end position B resulting in 100 different optimization problems. The perturbations are chosen to be uniformly distributed over the blue areas in Fig. \ref{fig:opt_trajectories}. The performance of FSLP and \texttt{Ipopt} is presented in Fig. \ref{fig:performancePlot}. \texttt{Ipopt} - feasible solution denotes the value at the iteration from which all subsequent iterations stay below the given feasibility tolerance. We can see that the SLP algorithm needs in almost all cases fewer outer iterations than \texttt{Ipopt} to solve the TOCP problems. In contrast, FSLP needs more constraint evaluations due to the inner feasibility iterations. Since FSLP needs fewer outer iterations than \texttt{Ipopt}, FSLP needs fewer evaluations of derivative information and especially no evaluations of the Hessian. Stopping FSLP at a suboptimal point satisfying start and end position constraints reduces the constraint evaluations and outer iterations. 

\begin{figure}[thpb]
\includegraphics[width=0.485\textwidth]{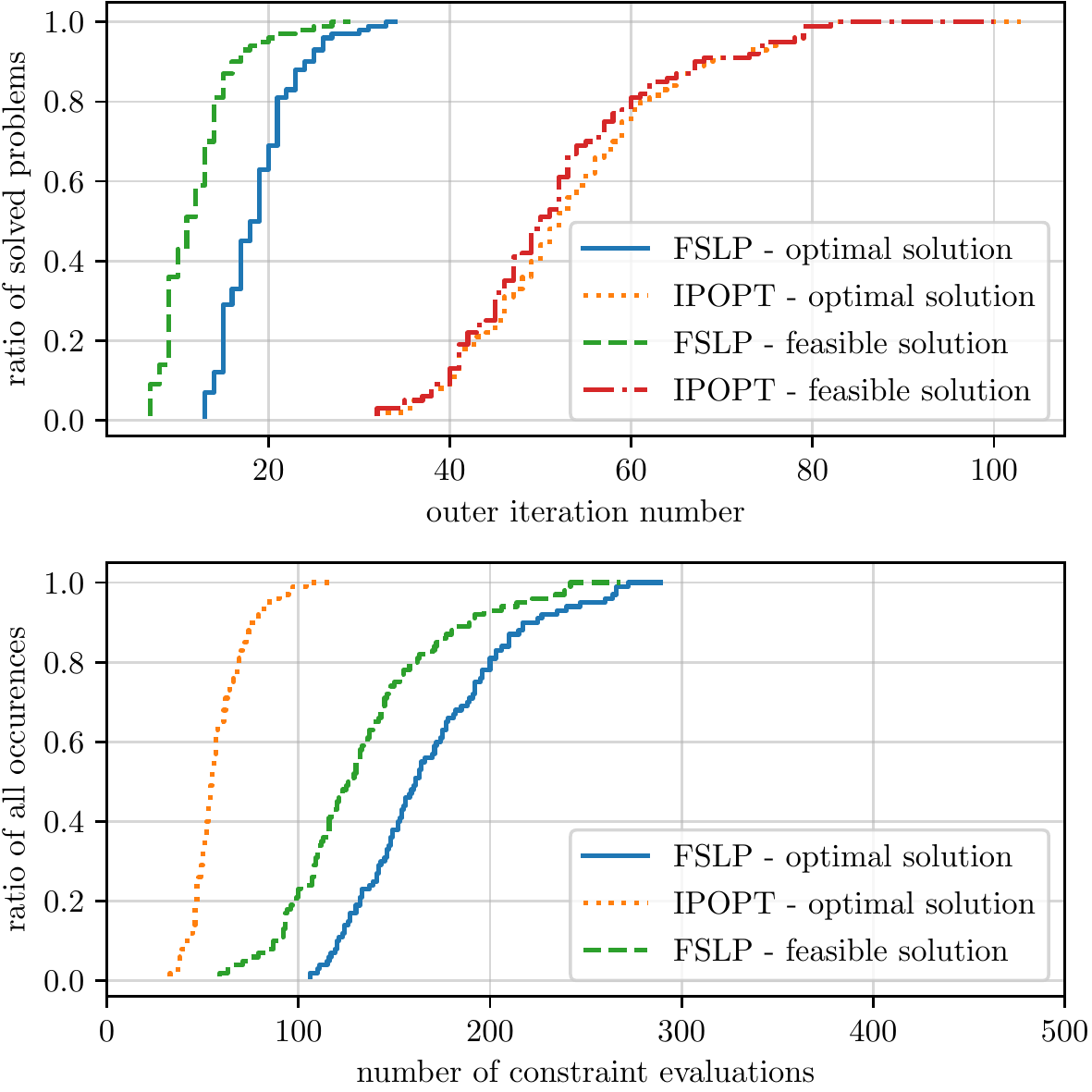}
\vspace{-5mm}
\caption{Comparison of outer iterations and constraint evaluations of FSLP and IPOPT on 100 different overhead crane problems.}
\label{fig:performancePlot}
\vspace{-7mm}
\end{figure}

Additionally, we timed the wall time of the Python FSLP prototype against \texttt{Ipopt} for solving an instance of the 100 different test problems. One big difference between FSLP and \texttt{Ipopt} is the use of derivative information, therefore we investigate the wall times on the 100 different problems for different numbers of internal discretization steps in the Runge Kutta scheme. The more steps are used, the more expensive it gets to evaluate the constraints and especially the derivatives. The results are illustrated in a boxplot in Fig. \ref{fig:walltimePlot}. We see that the run time of the prototypical implementation scales better than the run time of \texttt{Ipopt} for increasing numbers of discretization steps. These are promising results for future research and for an efficient implementation for real-time optimization by exploiting problem structure.

At last, we remark that the initial time was chosen close to the optimal time of the original overhead crane problem in Subsection \ref{subsec:overheadCrane}. If the initial time is chosen larger, then a zero-slack feasible solution can be reached faster due to more freedom in feasible trajectories, but the run time to find an optimal solution will in general increase.

\begin{figure}[thpb]
\includegraphics[width=0.485\textwidth]{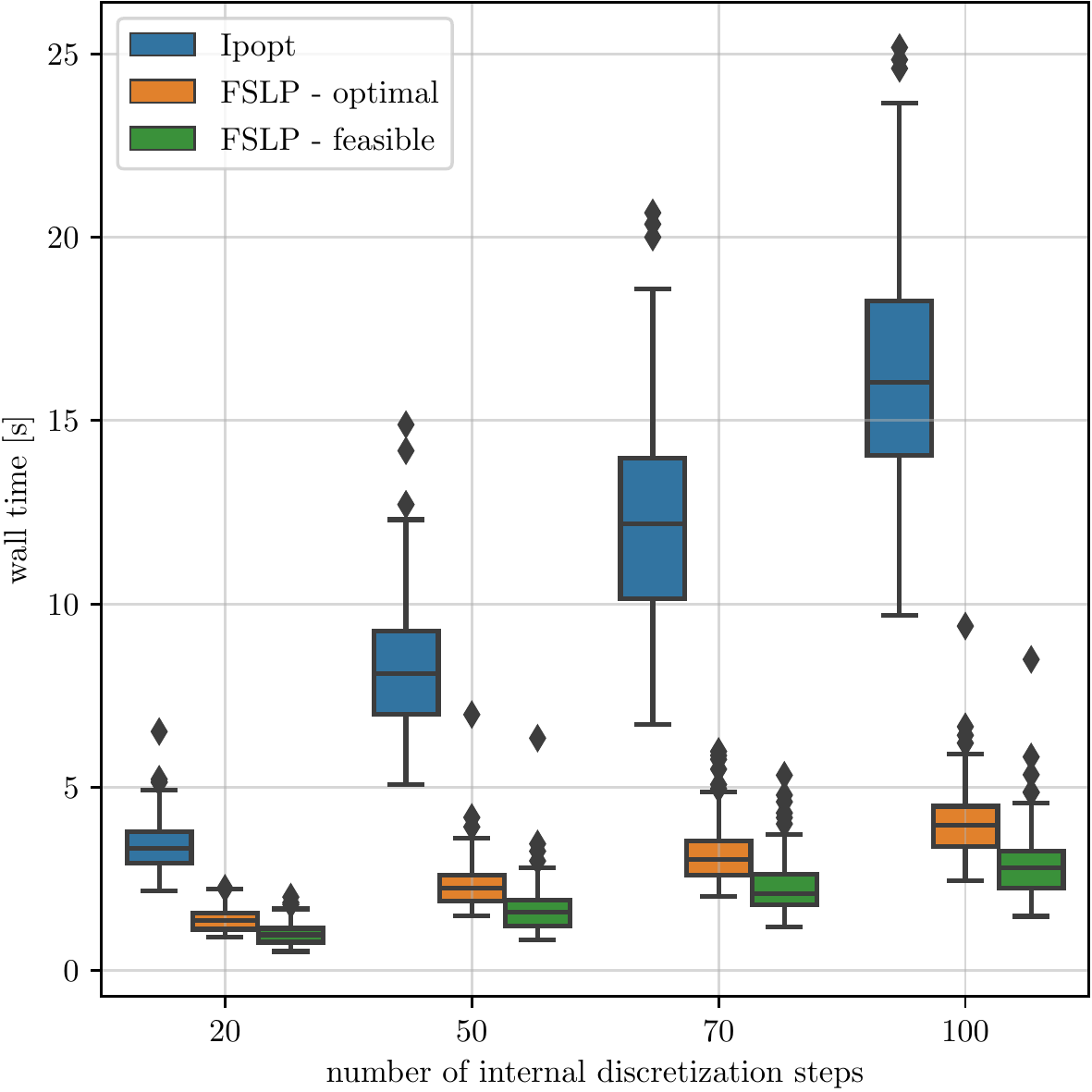}
\vspace{-5mm}
\caption{Comparison of wall times of FSLP and IPOPT on 100 different overhead crane problems for different number of discretization steps in the Runge Kutta scheme.}
\label{fig:walltimePlot}
\vspace{-6mm}
\end{figure}



\section{Conclusion}
\label{sec:conclusion}
\color{black}
In this paper, we proposed a novel globally convergent feasible sequential linear programming algorithm for time-optimal control problems. In the case of a fully determined system, we even obtain quadratic local convergence. We show that the algorithm can maintain feasible iterates also for problems with nonlinear constraints. In a numerical case study, the performance of the algorithm and its potential to stop iterations early was demonstrated. 

%


\bibliography{bibliography}

\begin{thebibliography}{10}

\bibitem{Andersson2019}
Joel A~E Andersson, Joris Gillis, Greg Horn, James~B Rawlings, and Moritz
  Diehl.
\newblock {CasADi} -- a software framework for nonlinear optimization and
  optimal control.
\newblock {\em Mathematical Programming Computation}, 11(1):1--36, 2019.

\bibitem{Bock2007}
H.~G. Bock, M.~Diehl, E.~A. Kostina, and J.~P. Schl\"oder.
\newblock Constrained optimal feedback control of systems governed by large
  differential algebraic equations.
\newblock In {\em Real-Time and Online PDE-Constrained Optimization}, pages
  3--22. {SIAM}, 2007.

\bibitem{Bock1984}
H.~G. Bock and K.~J. Plitt.
\newblock A multiple shooting algorithm for direct solution of optimal control
  problems.
\newblock In {\em Proceedings of the IFAC World Congress}, pages 242--247.
  Pergamon Press, 1984.

\bibitem{Boyd2004}
S.~Boyd and L.~Vandenberghe.
\newblock {\em Convex Optimization}.
\newblock University {P}ress, Cambridge, 2004.

\bibitem{Conn2000}
A.R. Conn, N.~Gould, and P.L. Toint.
\newblock {\em {T}rust-{R}egion {M}ethods}.
\newblock MPS/SIAM Series on Optimization. SIAM, Philadelphia, USA, 2000.

\bibitem{cplex2017v12}
IBM~ILOG Cplex.
\newblock V12.8: User’s manual for cplex.
\newblock {\em International Business Machines Corporation}, 2017.

\bibitem{Fletcher1982}
R.~Fletcher.
\newblock Second order corrections for non-differentiable optimization.
\newblock In G.~Alistair Watson, editor, {\em Numerical Analysis}, pages
  85--114, Berlin, Heidelberg, 1982. Springer Berlin Heidelberg.

\bibitem{Griffith1961}
Ro~E Griffith and RA~Stewart.
\newblock A nonlinear programming technique for the optimization of continuous
  processing systems.
\newblock {\em Management science}, 7(4):379--392, 1961.

\bibitem{HSL}
{HSL}.
\newblock {A} collection of {F}ortran codes for large scale scientific
  computation., 2011.

\bibitem{Kim2016}
Taedong Kim and Stephen~J. Wright.
\newblock An s l1 lp-active set approach for feasibility restoration in power
  systems.
\newblock {\em Optimization and Engineering}, 17:385--419, 6 2016.

\bibitem{Messerer2021}
Florian Messerer, Katrin Baumg{\"a}rtner, and Moritz Diehl.
\newblock Survey of sequential convex programming and generalized
  {G}auss-{N}ewton methods.
\newblock {\em ESAIM: Proceedings and Surveys}, 71:64--88, 2021.

\bibitem{Nocedal2006}
Jorge Nocedal and Stephen~J. Wright.
\newblock {\em Numerical Optimization}.
\newblock Springer Series in Operations Research and Financial Engineering.
  Springer, 2 edition, 2006.

\bibitem{Robinson1980}
Stephen~M. Robinson.
\newblock Strongly regular generalized equations.
\newblock {\em Mathematics of Operations Research}, 5:43--62, 2 1980.

\bibitem{Tenny2004}
M.J. Tenny, S.J. Wright, and J.B. Rawlings.
\newblock {N}onlinear model predictive control via feasibility-perturbed
  sequential quadratic programming.
\newblock {\em Computational Optimization and Applications}, 28:87--121, 2004.

\bibitem{Waechter2006}
Andreas W\"achter and Lorenz~T. Biegler.
\newblock On the implementation of an interior-point filter line-search
  algorithm for large-scale nonlinear programming.
\newblock {\em Mathematical Programming}, 106(1):25--57, 2006.

\bibitem{Wright2004}
Stephen~J. Wright and Matthew~J. Tenny.
\newblock {A} feasible trust-region sequential quadratic programming algorithm.
\newblock {\em SIAM Journal on Optimization}, 14:1074--1105, 1 2004.

\bibitem{Zanelli2021}
Andrea Zanelli.
\newblock {\em Inexact methods for nonlinear model predictive control:
  stability, applications, and software}.
\newblock PhD thesis, University of Freiburg, 2021.

\end{thebibliography}
\bibliographystyle{plain} 

%
%
%

\end{document}